\documentclass[12pt]{article}
\usepackage{amsmath,amsthm,amssymb}
\usepackage{amssymb,latexsym}
\newtheorem{theorem}{Theorem}
\newtheorem{conjecture}{Conjecture}
\newtheorem{lemma}{Lemma}

\begin{document}

\baselineskip=17pt

\title{\bf A tangent inequality over primes}

\author{\bf S. I. Dimitrov}
\date{2021}
\maketitle
\begin{abstract}
In this paper we introduce a new diophantine inequality with prime numbers.
Let $1<c<\frac{10}{9}$.
We show that for any fixed $\theta>1$, every sufficiently large positive number $N$
and a small constant $\varepsilon>0$, the tangent inequality
\begin{equation*}
\big|p^c_1\tan^\theta(\log p_1)+  p^c_2\tan^\theta(\log p_2)+ p^c_3\tan^\theta(\log p_3) -N\big|<\varepsilon
\end{equation*}
has a solution in prime numbers $p_1,\,p_2,\,p_3$.
\quad\\
{\bf  2020 Math.\ Subject Classification}:  11J25 $\cdot$ 11P55 $\cdot$ 11L07
\end{abstract}

\section{Introduction and statements of the result}
\indent

The ternary Piatetski-Shapiro inequality is a major diophantine approximation of analytic number theory.
It is the diophantine inequality
\begin{equation}\label{Tolevinequality}
|p_1^c+p_2^c+p_3^c-N|<\varepsilon\,,
\end{equation}
where $N$ is a sufficiently large positive number,
$p_1,\,p_2,\,p_3$ are prime numbers, $c>1$ is not an integer and $\varepsilon>0$ is a small constant.
In 1992 Tolev \cite{Tolev1} proved that \eqref{Tolevinequality} has a solution for $1<c<\frac{15}{14}$.
Subsequently the result of Tolev was improved by Cai \cite{Cai1} to $1<c<\frac{13}{12}$,
by Cai \cite{Cai2} and Kumchev and Nedeva \cite{Ku-Ne} to $1<c<\frac{12}{11}$,
by Cao and Zhai \cite{Cao-Zhai} to $1<c<\frac{237}{214}$,
by Kumchev \cite{Kumchev} to $1<c<\frac{61}{55}$\,,
by Baker and Weingartner \cite{Baker-Weingartner} to $1<c<\frac{10}{9}$,
by Cai \cite{Cai3} to $1<c<\frac{43}{36}$,
by Baker \cite{Baker} to $1<c<\frac{6}{5}$ and this is the best result up to now.

Motivated by these results, recently the author \cite{Dimitrov} proved that for every
sufficiently large positive number $N$  and a small constant $\varepsilon>0$,
the logarithmic inequality
\begin{equation*}
\big|p_1\log p_1+p_2\log p_2+p_3\log p_3-N\big|<\varepsilon
\end{equation*}
has a solution in prime numbers $p_1,\,p_2,\,p_3$.

In this paper we introduce a new  diophantine inequality with prime numbers.
More precisely we prove the following theorem.
\begin{theorem} Let $1<c<\frac{10}{9}$.
For any fixed $\theta>1$ and every sufficiently large positive number $N$ the tangent inequality
\begin{equation*}
\big|p^c_1\tan^\theta(\log p_1)+  p^c_2\tan^\theta(\log p_2)+ p^c_3\tan^\theta(\log p_3) -N\big|
<\left(\frac{2^\theta N}{3^{\theta+3}}\right)^{\frac{1}{c^2}\left(c-\frac{10}{9}\right)}
\end{equation*}
is solvable in prime numbers $p_1,\,p_2,\,p_3$.
\end{theorem}
As usual the corresponding binary problem is out of reach of the current state of the analytic number theory.
In other words we have the following conjecture.
\begin{conjecture} Let $N$ is a sufficiently large positive number, $\varepsilon>0$ is a small constant and $\theta>1$ is fixed.
There exists $c_0>1$ such that  for any fixed $1<c<c_0$, the tangent  inequality
\begin{equation*}
\big|p^c_1\tan^\theta(\log p_1)+  p^c_2\tan^\theta(\log p_2) -N\big|<\varepsilon
\end{equation*}
is solvable in prime numbers $p_1,\,p_2$.
\end{conjecture}
The next conjecture is maybe more difficult.
\begin{conjecture} Let $N$ is a sufficiently large positive number and $\varepsilon>0$ is a small constant.
There exists $c_0>1$ such that  for any fixed $1<c<c_0$, the tangent inequality
\begin{equation*}
\big|\tan^cp_1 + \tan^cp_2 + \tan^cp_3-N\big|<\varepsilon
\end{equation*}
is solvable in prime numbers $p_1,\,p_2,\,p_3$.
\end{conjecture}

\section{Notations}
\indent

Assume that $N$ is a sufficiently large  positive number.
The letter $\eta$ denotes an arbitrary small positive number, not the same in all appearances.
As usual $\Lambda(n)$ is von Mangoldt's function.
We denote by $[y]$ the integer part of $y$. Moreover $e(y)=e^{2\pi i y}$.
The letter $p$  with or without subscript will always denote prime number.
We denote by $\tau _k(n)$ the number of solutions of the equation $m_1m_2\ldots m_k$ $=n$ in natural numbers $m_1,\,\ldots,m_k$.
Throughout this paper we suppose that $1<c<\frac{10}{9}$.
Assume that $\theta>1$ is a fixed. Let $X$ is an arbitrary solution of the equation
\begin{equation}\label{XN}
\pi\left[\frac{\log X}{\pi}\right]+\arctan\frac{3}{2}=\frac{1}{c}\log\frac{2^\theta N}{3^{\theta+1}} \; .
\end{equation}
Denote
\begin{align}
\label{varepsilon}
&\varepsilon=X^{\frac{1}{c}\left(c-\frac{10}{9}\right)}\,;\\
\label{tau}
&\tau= X^{\frac{1}{9}-c}\,;\\
\label{H}
&H=X^{\frac{10}{9}-c}\,;\\
\label{Delta1}
&\Delta_1=e^{\pi\big[\frac{\log X}{\pi}\big]+\arctan\frac{4}{9}}\,;\\
\label{Delta2}
&\Delta_2=e^{\pi\big[\frac{\log X}{\pi}\big]+\arctan2}\,;\\
\label{Salpha}
&S(\alpha)=\sum\limits_{\Delta_1<p\leq \Delta_2} e\big(\alpha p^c\tan^\theta(\log p)\big)\log p\,;\\
\label{Int}
&I(\alpha)=\int\limits_{\Delta_1}^{\Delta_2}e\big(\alpha y^c\tan^\theta(\log y)\big) \,dy\,.
\end{align}

\section{Lemmas}
\indent

\begin{lemma}\label{Fourier}Let $k\in \mathbb{N}$.
There exists a function $\psi(y)$ which is $k$ times continuously differentiable and
such that
\begin{align*}
  &\psi(y)=1\quad\quad\quad\mbox{for }\quad\quad|y|\leq 3\varepsilon/4\,;\\[6pt]
  &0\leq\psi(y)<1\quad\mbox{for}\quad3\varepsilon/4 <|y|< \varepsilon\,;\\[6pt]
  &\psi(y)=0\quad\quad\quad\mbox{for}\quad\quad|y|\geq \varepsilon\,.
\end{align*}
and its Fourier transform
\begin{equation*}
\Psi(x)=\int\limits_{-\infty}^{\infty}\psi(y)e(-xy)dy
\end{equation*}
satisfies the inequality
\begin{equation*}
|\Psi(x)|\leq\min\bigg(\frac{7\varepsilon}{4},\frac{1}{\pi|x|},\frac{1}{\pi |x|}
\bigg(\frac{k}{2\pi |x|\varepsilon/8}\bigg)^k\bigg)\,.
\end{equation*}
\end{lemma}
\begin{proof}
See (\cite{Shapiro} or \cite{Segal}).
\end{proof}

\begin{lemma}\label{GrahamandKolesnik}

Let $k \geq0$ be an integer.
Suppose that $f(t)$ has $k+2$ continuous derivatives on $I$, and that $I \subseteq(N,2N]$.
Assume also that there is some constant $F$ such that
\begin{equation}\label{frFNR}
|f^{(r)}(t)|\asymp F N^{-r}
\end{equation}
for $r = 1, \ldots, k + 2$. Let $Q = 2^k$. Then
\begin{equation*}
\bigg|\sum_{n\in I}e(f(n))\bigg|\ll F^{\frac{1}{4Q-2}} N^{1-\frac{k+2}{4Q-2}}  +F^{-1} N\,.
\end{equation*}
The implied constant depends only upon the implied constants in \eqref{frFNR}.
\end{lemma}
\begin{proof}
See (\cite{Graham-Kolesnik}, Theorem 2.9).
\end{proof}

\begin{lemma}\label{Squareoutlemma}
For any complex numbers $a(n)$ we have
\begin{equation*}
\bigg|\sum_{a<n\le b}a(n)\bigg|^2
\leq\bigg(1+\frac{b-a}{Q}\bigg)\sum_{|q|< Q}\bigg(1-\frac{|q|}{Q}\bigg)
\sum_{a<n,\, n+q\leq b}a(n+q)\overline{a(n)},
\end{equation*}
where $Q$ is any positive integer.
\end{lemma}
\begin{proof}
See (\cite{Iwaniec-Kowalski}, Lemma 8.17).
\end{proof}

\begin{lemma}\label{Iest}
Assume that $F(x)$, $G(x)$ are real functions defined in  $[a,b]$,
$|G(x)|\leq H$ for $a\leq x\leq b$ and $G(x)/F'(x)$ is a monotonous function. Set
\begin{equation*}
I=\int\limits_{a}^{b}G(x)e(F(x))dx\,.
\end{equation*}
If $F'(x)\geq h>0$ for all $x\in[a,b]$ or if $F'(x)\leq-h<0$ for all $x\in[a,b]$ then
\begin{equation*}
|I|\ll H/h\,.
\end{equation*}
If $F''(x)\geq h>0$ for all $x\in[a,b]$ or if $F''(x)\leq-h<0$ for all $x\in[a,b]$ then
\begin{equation*}
|I|\ll H/\sqrt h\,.
\end{equation*}
\end{lemma}
\begin{proof}
See (\cite{Titchmarsh}, p. 71).
\end{proof}

\begin{lemma}\label{Heath-Brown} Let $G(n)$ be a complex valued function.
Assume further that
\begin{align*}
&P>2\,,\quad P_1\le 2P\,,\quad  2\le U<V\le Z\le P\,,\\
&U^2\le Z\,,\quad 128UZ^2\le P_1\,,\quad 2^{18}P_1\le V^3\,.
\end{align*}
Then the sum
\begin{equation*}
\sum\limits_{P<n\le P_1}\Lambda(n)G(n)
\end{equation*}
can be decomposed into $O\Big(\log^6P\Big)$ sums, each of which is either of Type I
\begin{equation*}
\mathop{\sum\limits_{M<m\le M_1}a_m\sum\limits_{L<l\le L_1}}_{P<ml\le P_1}G(ml)
\end{equation*}
and
\begin{equation*}
\mathop{\sum\limits_{M<m\le M_1}a_m\sum\limits_{L<l\le L_1}}_{P<ml\le P_1}G(ml)\log l\,,
\end{equation*}
where
\begin{equation*}
L\ge Z\,,\quad M_1\le 2M\,,\quad L_1\le 2L\,,\quad a_m\ll \tau _5(m)\log P
\end{equation*}
or of Type II
\begin{equation*}
\mathop{\sum\limits_{M<m\le M_1}a_m\sum\limits_{L<l\le L_1}}_{P<ml\le P_1}b_lG(ml)
\end{equation*}
where
\begin{equation*}
U\le L\le V\,,\quad M_1\le 2M\,,\quad L_1\le 2L\,,\quad
a_m\ll \tau _5(m)\log P\,,\quad b_l\ll \tau _5(l)\log P\,.
\end{equation*}
\end{lemma}

\begin{proof}
See (\cite{Heath}).
\end{proof}

\begin{lemma}\label{SIlemma}
If $|\alpha|\leq\tau$ then
\begin{equation*}
S(\alpha)=I(\alpha)+\mathcal{O}\Big(Xe^{-(\log X)^{1/5}}\Big)\,.
\end{equation*}
\end{lemma}
\begin{proof}
This lemma is very similar to result of Tolev \cite{Tolev1}.
Inspecting the arguments presented in (\cite{Tolev1}, Lemma 14), (with $T=X^{\frac{1}{5}}$)
the reader will easily see that the proof of Lemma \ref{SIlemma} can be obtained by the same way.
\end{proof}

\begin{lemma}\label{Thetaest}
We have
\begin{equation*}
\int\limits_{-\infty}^{\infty}I^3(\alpha)\Psi(\alpha)e(-N\alpha)\,d\alpha
\gg\varepsilon X^{3-c}\,.
\end{equation*}
\end{lemma}
\begin{proof}
Denoting the above integral with $\Theta$, using \eqref{Int}, the definition of $\psi(y)$
and the inverse Fourier transformation formula we obtain
\begin{align}\label{Theta}
\Theta&=\int\limits_{\Delta_1}^{\Delta_2}\int\limits_{\Delta_1}^{\Delta_2}\int\limits_{\Delta_1}^{\Delta_2}\int\limits_{-\infty}^{\infty}
\Psi(\alpha)e\big(\big(y_1^c\tan^\theta(\log y_1)+y_2^c\tan^\theta(\log y_2)+y_3^c\tan^\theta(\log y_3)-N\big)\alpha\big)\,d\alpha \,dy_1\,dy_2\,dy_3\nonumber\\
&=\int\limits_{\Delta_1}^{\Delta_2}\int\limits_{\Delta_1}^{\Delta_2}\int\limits_{\Delta_1}^{\Delta_2}
\psi\big(y_1^c\tan^\theta(\log y_1)+y_2^c\tan^\theta(\log y_2)+y_3^c\tan^\theta(\log y_3)-N\big)\,dy_1\,dy_2\,dy_3\nonumber\\
&\geq\mathop{\int\limits_{\Delta_1}^{\Delta_2}\int\limits_{\Delta_1}^{\Delta_2}\int\limits_{\Delta_1}^{\Delta_2}}
_{\substack{|y_1^c\tan^\theta(\log y_1)+y_2^c\tan^\theta(\log y_2)+y_3^c\tan^\theta(\log y_3)-N|<3\varepsilon/4}}\,dy_1\,dy_2\,dy_3\nonumber\\
&\geq\int\limits_{\Delta_\lambda}^{\Delta_\mu}\int\limits_{\Delta_\lambda}^{\Delta_\mu}
\left(\int\limits_{\Omega}\,dy_3\right)\,dy_1\,dy_2\,,
\end{align}
where $\lambda$ and $\mu$ are real numbers such that
\begin{equation}\label{lambdamu}
\frac{4}{9}<\frac{3}{2}<\left(\frac{2}{5}\left(\frac{3^{\theta+1}}{2^\theta}-\frac{3}{4}\right)\right)^{\frac{1}{\theta}}
<\lambda<\mu<\left(\frac{2}{5}\left(\frac{3^{\theta+1}}{2^\theta}-\frac{2}{3}\right)\right)^{\frac{1}{\theta}}<2\,,
\end{equation}
\begin{align}
\label{Deltalambda}
&\Delta_\lambda=e^{\pi\big[\frac{\log X}{\pi}\big]+\arctan\lambda}\,,\\
\label{Deltamu}
&\Delta_\mu=e^{\pi\big[\frac{\log X}{\pi}\big]+\arctan\mu}
\end{align}
and
\begin{equation}\label{Omega1}
\Omega=[\Delta_1, \Delta_2]\cap\big[y'_3, y''_3\big]\,,
\end{equation}
where the interval $\big[y'_3, y''_3\big]$ is a solution of the system inequalities
\begin{equation}\label{System1}
\left|\begin{array}{cc}
y_3^c\tan^\theta(\log y_3)>N-3\varepsilon/4-y_1^c\tan^\theta(\log y_1)-y_2^c\tan^\theta(\log y_2)\\
y_3^c\tan^\theta(\log y_3)<N+3\varepsilon/4-y_1^c\tan^\theta(\log y_1)-y_2^c\tan^\theta(\log y_2\;\;
\end{array}\right..
\end{equation}
From \eqref{XN}, \eqref{varepsilon}, \eqref{Delta1},  \eqref{Delta2}, \eqref{lambdamu},
\eqref{Deltalambda}, \eqref{Deltamu} and \eqref{System1}
it follows that
\begin{equation}\label{subsetDelta}
\big[y'_3, y''_3\big]\subset[\Delta_1, \Delta_2]\,.
\end{equation}
Now \eqref{Omega1} and \eqref{subsetDelta} yield
\begin{equation}\label{Omega2}
\Omega=\big[y'_3, y''_3\big]\,.
\end{equation}
Consider the function $t(y)$ defined by
\begin{equation}\label{Implicitfunction}
t=y^c\tan^\theta(\log y)\,,
\end{equation}
for
\begin{equation}\label{yy'y''}
y\in\big[y'_3, y''_3\big]\,.
\end{equation}
The first derivative of $y$ as implicit function of $t$  is
\begin{equation}\label{Firstderivative}
y'=\frac{y^{1-c}}{\big(c\tan(\log y)+\theta\sec^2(\log y)\big)\tan^{\theta-1}(\log y)}\,.
\end{equation}
By  \eqref{Delta1},  \eqref{Delta2}, \eqref{subsetDelta}, \eqref{yy'y''} and \eqref{Firstderivative} we deduce
\begin{equation}\label{y'est1}
y'\asymp X^{1-c}\,.
\end{equation}
Using \eqref{Theta}, \eqref{System1}, \eqref{Omega2}, \eqref{Implicitfunction} and the mean-value theorem we get
\begin{equation}\label{Thetaest1}
\Theta\gg\varepsilon\int\limits_{\Delta_\lambda}^{\Delta_\mu}\int\limits_{\Delta_\lambda}^{\Delta_\mu}
y'\big(\xi_{y_1,y_2}\big)\,dy_1\,dy_2\,,
\end{equation}
where
\begin{equation*}
\xi_{y_1,y_2}\asymp X^c\,.
\end{equation*}
Bearing in mind \eqref{Deltalambda}, \eqref{Deltamu}, \eqref{y'est1} and \eqref{Thetaest1} we obtain
\begin{equation*}
\Theta\gg\varepsilon X^{3-c}\,.
\end{equation*}
The lemma is proved.
\end{proof}

\begin{lemma}\label{3Int}
We have
\begin{align*}
&\emph{(i)}\quad\quad\quad\quad\int\limits_{-\tau}^\tau|S(\alpha)|^2\,d\alpha\ll X^{2-c}\log^3X\,,
\quad\quad\quad\quad\quad\quad\quad\quad\quad\quad\quad\quad\quad\quad\quad\quad\\
&\emph{(ii)}\quad\quad\quad\quad\int\limits_{-\tau}^\tau|I(\alpha)|^2\,d\alpha\ll X^{2-c}\log X\,,
\quad\quad\quad\quad\quad\quad\quad\quad\quad\quad\quad\quad\quad\quad\quad\quad\\
&\emph{(iii)}\quad\quad\quad\quad\int\limits_{n}^{n+1}|S(\alpha)|^2\,d\alpha\ll X\log^3X\,.
\quad\quad\quad\quad\quad\quad\quad\quad\quad\quad\quad\quad\quad\quad\quad\quad
\end{align*}
\end{lemma}

\begin{proof}
We only prove  $\textmd{(i)}$. The cases $\textmd{(ii)}$ and $\textmd{(iii)}$  are analogous.

From \eqref{Salpha}  we write
\begin{align}\label{Squareout}
\int\limits_{-\tau}^\tau|S(\alpha)|^2\,d\alpha&=\sum\limits_{\Delta_1<p_1,p_2\leq\Delta_2}\log p_1\log p_2
\int\limits_{-\tau}^\tau e\big(\big(p_1^c\tan^\theta(\log p_1)-p_2^c\tan^\theta(\log p_2)\big)\alpha\big)\,d\alpha\nonumber\\
&\ll\sum\limits_{\Delta_1<p_1,p_2\leq\Delta_2}\log p_1\log p_2
\min\bigg(\tau,\frac{1}{|p_1^c\tan^\theta(\log p_1)-p_2^c\tan^\theta(\log p_2)|}\bigg)\nonumber\\
&\ll\tau\sum\limits_{\Delta_1<p_1,p_2\leq\Delta_2\atop{|p_1^c\tan^\theta(\log p_1)-p_2^c\tan^\theta(\log p_2)|\leq1/\tau}}\log p_1\log p_2\nonumber\\
&+\sum\limits_{\Delta_1<p_1,p_2\leq\Delta_2\atop{|p_1^c\tan^\theta(\log p_1)-p_2^c\tan^\theta(\log p_2)|>1/\tau}}
\frac{\log p_1\log p_2}{|p_1^c\tan^\theta(\log p_1)-p_2^c\tan^\theta(\log p_2)|}\nonumber\\
&\ll U\tau\log^2X+V\log^2X,
\end{align}
where
\begin{align}
\label{U}
&U=\sum\limits_{\Delta_1<n_1,n_2\leq\Delta_2\atop{|n_1^c\tan^\theta(\log n_1)-n_2^c\tan^\theta(\log n_2)|\leq1/\tau}}1\,,\\
\label{V}
&V=\sum\limits_{\Delta_1<n_1,n_2\leq\Delta_2\atop{|n_1^c\tan^\theta(\log n_1)-n_2^c\tan^\theta(\log n_2)|>1/\tau}}
\frac{1}{|n_1^c\tan^\theta(\log n_1)-n_2^c\tan^\theta(\log n_2)|}\,.
\end{align}
Let the interval $\big[n'_2, n''_2\big]$ is a solution of the system inequalities
\begin{equation}\label{System2}
\left|\begin{array}{ccc}
n_2\in[\Delta_1, \Delta_2] \quad\quad\quad\quad\quad\quad\quad\quad\quad\quad\quad \\
n_2^c\tan^\theta(\log n_2)\geq n_1^c\tan^\theta(\log n_1)-1/\tau\\
n_2^c\tan^\theta(\log n_2)\leq n_1^c\tan^\theta(\log n_1)+1/\tau
\end{array}\right..
\end{equation}
Consider the function $t(y)$ defined by \eqref{Implicitfunction} for
\begin{equation}\label{yn'n''}
y\in\big[n'_2, n''_2\big]\,.
\end{equation}
By   \eqref{Delta1},  \eqref{Delta2}, \eqref{Implicitfunction}, \eqref{Firstderivative}, \eqref{System2} and \eqref{yn'n''}
for the first derivative of $y$ as implicit function of $t$ we obtain
\begin{equation}\label{y'est2}
y'\asymp X^{1-c}\,.
\end{equation}
Using \eqref{Implicitfunction}, \eqref{U}, \eqref{System2} and the mean-value theorem we get
\begin{equation}\label{Uest1}
U\ll\mathop{\sum\limits_{\Delta_1<n_1\leq\Delta_2}\sum\limits_{\Delta_1<n_2\leq\Delta_2}}_
{n_1^c\tan^\theta(\log n_1)-1/\tau\leq n_2^c\tan^\theta(\log n_2)\leq n_1^c\tan^\theta(\log n_1)+1/\tau}1
\ll X + \frac{1}{\tau}\sum\limits_{\Delta_1<n_1\leq\Delta_2}y'(\xi)\,,
\end{equation}
where $\xi\asymp X^c$.
Taking into account  \eqref{Delta1},  \eqref{Delta2}, \eqref{y'est2} and \eqref{Uest1} we find
\begin{equation}\label{Uest2}
U\ll X+\frac{X^{2-c}}{\tau}\,.
\end{equation}
On the other hand from \eqref{V} we have
\begin{equation}\label{VVl}
V\leq\sum\limits_{l}V_l\,,
\end{equation}
where
\begin{equation*}
V_l=\sum\limits_{\Delta_1<n_1,n_2\leq\Delta_2\atop{l<|n_1^c\tan^\theta(\log n_1)-n_2^c\tan^\theta(\log n_2)|\leq2l}}
\frac{1}{|n_1^c\tan^\theta(\log n_1)-n_2^c\tan^\theta(\log n_2)|}
\end{equation*}
and $l$ takes the values $2^k/\tau,\,k=0,1,2,...$, with $l\ll X^c$.\\
Arguing as in $U$  and  using the mean-value theorem we deduce
\begin{align}\label{Vlest}
V_l&\ll\frac{1}{l}\mathop{\sum\limits_{\Delta_1<n_1\leq\Delta_2}\sum\limits_{\Delta_1<n_2\leq\Delta_2}}_
{n_1^c\tan^\theta(\log n_1)+l\leq n_2^c\tan^\theta(\log n_2)\leq n_1^c\tan^\theta(\log n_1)+2l}1\nonumber\\
&+\frac{1}{l}\mathop{\sum\limits_{\Delta_1<n_1\leq\Delta_2}\sum\limits_{\Delta_1<n_2\leq\Delta_2}}_
{n_1^c\tan^\theta(\log n_1)-2l\leq n_2^c\tan^\theta(\log n_2)\leq n_1^c\tan^\theta(\log n_1)-l}1\nonumber\\
&\ll\sum\limits_{n_1\in\Delta}y'(\xi)\nonumber\\
&\ll X^{2-c}\,.
\end{align}
The proof follows from \eqref{tau}, \eqref{Squareout}, \eqref{Uest2}, \eqref{VVl} and \eqref{Vlest}.
\end{proof}

\begin{lemma}\label{SIest} Assume that
\begin{equation}\label{DeltaalphaH}
\tau \leq |\alpha| \leq H \,.
\end{equation}
Set
\begin{equation}\label{SI}
S_I=\mathop{\sum\limits_{M<m\le M_1}a_m\sum\limits_{L<l\le L_1}}_{\Delta_1<ml\le \Delta_2}
e\big(\alpha m^cl^c \tan^\theta\big(\log (ml) \big)\big)
\end{equation}
and
\begin{equation}\label{SI'}
S'_I=\mathop{\sum\limits_{M<m\le M_1}a_m\sum\limits_{L<l\le L_1}}_{\Delta_1<ml\le \Delta_2}
e\big(\alpha m^cl^c \tan^\theta\big(\log (ml) \big)\big)\log l\,,
\end{equation}
where
\begin{equation}\label{Conditions1}
L\ge X^{\frac{13}{27}}\,,\quad M_1\le 2M\,,\quad L_1\le 2L\,,\quad a_m\ll \tau _5(m)\log \Delta_1\,.
\end{equation}
Then
\begin{equation*}
S_I\ll X^{\frac{17}{18}+\eta}\,.
\end{equation*}
\end{lemma}
\begin{proof}
First we notice that \eqref{Delta1}, \eqref{Delta2}, \eqref{SI}  and \eqref{Conditions1} imply
\begin{equation}\label{LMasympX}
LM\asymp X\,.
\end{equation}
Denote
\begin{equation}\label{flm}
f(l, m)=\alpha m^cl^c \tan^\theta\big(\log (ml) \big)\,.
\end{equation}
By  \eqref{Delta1}, \eqref{SI}, \eqref{Conditions1} and \eqref{flm}  we write
\begin{equation}\label{SIest1}
S_I\ll X^\eta\sum\limits_{M<m\leq M_1}\bigg|\sum\limits_{L'<l\leq L'_1}e\big(f(l, m)\big)\bigg|\,,
\end{equation}
where
\begin{equation}\label{L'L1'}
L'=\max{\bigg\{L,\frac{\Delta_1}{m}\bigg\}}\,,\quad L_1'=\min{\bigg\{L_1,\frac{\Delta_2}{m}\bigg\}}\,.
\end{equation}
From  \eqref{Conditions1} and \eqref{L'L1'} for the sum in \eqref{SIest1} it follows
\begin{equation}\label{L'andL1'inL2L}
\left|\begin{array}{cccc}
M<m\leq M_1\quad \; \\
L'<l\leq L'_1  \quad \quad \\
\Delta_1<ml\le \Delta_2 \;\;\\
(L', L'_1]\subseteq (L, 2L]
\end{array}\right..
\end{equation}
On the other hand for the function defined by  \eqref{flm} we find
\begin{equation}\label{firstderivativel}
\frac{\partial f(l, m)}{\partial l}=\alpha m^cl^{c-1}\tan^{{\theta}-1}\big(\log(ml)\big)
\Big(c\tan\big(\log(ml)\big))+{\theta}\sec^2\big(\log(ml)\big)\Big)
\end{equation}
and
\begin{align}
\label{secondderivativel}
\frac{\partial^2f(l, m)}{\partial l^2} &=\alpha m^cl^{c-2}\tan^{\theta-2}\big(\log(ml)\big)
\Big(\big(2\theta\sec^2\big(\log(ml)\big)+c^2-c\big)\tan^2\big(\log(ml)\big)\nonumber\\
&+(2c-1)\theta\sec^2\big(\log(ml)\big)\tan\big(\log(ml)\big)+(\theta^2-\theta)\sec^4\big(\log(ml)\big)\Big)\,.
\end{align}
Now \eqref{Conditions1}, \eqref{L'andL1'inL2L}, \eqref{firstderivativel} and \eqref{secondderivativel} yields
\begin{equation}\label{firstderivativelasymp}
\frac{\partial f(d,l)}{\partial l}\asymp |\alpha| M^cL^{c-1}
\end{equation}
and
\begin{equation}\label{secondderivativelasymp}
\frac{\partial^2f(d,l)}{\partial l^2}\asymp |\alpha| M^cL^{c-2}\,.
\end{equation}
Proceeding in the same way we get
\begin{equation}\label{thirdderivativelasymp}
\frac{\partial^3f(d,l)}{\partial l^3}\asymp |\alpha| M^cL^{c-3}\,.
\end{equation}
Using \eqref{tau}, \eqref{H}, \eqref{DeltaalphaH}, \eqref{Conditions1},
\eqref{LMasympX}, \eqref{SIest1}, \eqref{L'andL1'inL2L}, \eqref{firstderivativelasymp}, \eqref{secondderivativelasymp},
\eqref{thirdderivativelasymp} and Lemma \ref{GrahamandKolesnik} with  $k=1$ we obtain
\begin{align*}\label{SIest1}
S_I&\ll X^\eta\sum\limits_{M<m\leq M_1}\bigg(|\alpha|^{\frac{1}{6}}M^{\frac{c}{6}} L^{\frac{c}{6}}L^{\frac{1}{2}}
+|\alpha|^{-1}M^{-c}L^{1-c}\bigg)\\
&\ll X^\eta\Big(M^{\frac{1}{2}}H^{\frac{1}{6}}X^{\frac{c}{6}+\frac{1}{2}}+\tau^{-1}X^{1-c}\Big)\\
&\ll X^\eta\Big(M^{\frac{1}{2}}X^{\frac{37}{54}}+X^{\frac{8}{9}}\Big)\\
&\ll  X^{\frac{17}{18}+\eta}\,.
\end{align*}
To estimate the sum defined by \eqref{SI'} we apply Abel's summation formula and proceed in the same way to deduce
\begin{equation*}
S_I'\ll  X^{\frac{17}{18}+\eta}\,.
\end{equation*}
This proves the lemma.
\end{proof}

\begin{lemma}\label{SIIest} Assume that
\begin{equation}\label{DeltaalphaH2}
\tau \leq |\alpha| \leq H \,.
\end{equation}
Set
\begin{equation}\label{SII}
S_{II}=\mathop{\sum\limits_{M<m\le M_1}a_m\sum\limits_{L<l\le L_1}}_{\Delta_1<ml\le \Delta_2}b_l
e\big(\alpha m^cl^c \tan^\theta\big(\log (ml) \big)\big)\,,
\end{equation}
where
\begin{equation}
\begin{split}\label{Conditions2}
&2^{-11}X^{\frac{1}{27}}\leq L\leq 2^7X^{\frac{1}{3}}\,,\quad M_1\le 2M\,,\quad L_1\le 2L\,,\\
&a_m\ll \tau _5(m)\log \Delta_1\,,\quad b_l\ll \tau _5(l)\log \Delta_1\,.
\end{split}
\end{equation}
Then
\begin{equation*}
S_{II}\ll  X^{\frac{17}{18}+\eta}\,.
\end{equation*}
\end{lemma}
\begin{proof}
First we notice that \eqref{Delta1}, \eqref{Delta2}, \eqref{SII}  and \eqref{Conditions2} give us
\begin{equation}\label{LMasympX2}
LM\asymp X\,.
\end{equation}
From \eqref{Delta1}, \eqref{flm}, \eqref{SII}, \eqref{Conditions2}, \eqref{LMasympX2}, Cauchy's inequality
and Lemma \ref{Squareoutlemma} with $Q=X^{\frac{1}{9}}$ it follows
\begin{align}\label{SIIest1}
S_{II}&\ll\left(\sum\limits_{M<m\le M_1}|a(m)|^2\right)^{\frac{1}{2}}
\left(\sum\limits_{M<m\le M_1}\bigg|\sum\limits_{L<l\le L_1\atop{\Delta_1<ml\le \Delta_2}}
b(l)e\big(f(l, m)\big)\bigg|^2\right)^{\frac{1}{2}}\nonumber\\
&\ll M^{\frac{1}{2}+\eta}\left(\sum\limits_{M<m\le M_1}\frac{L}{Q}\sum_{|q|<Q}\bigg(1-\frac{q}{Q}\bigg)
\sum\limits_{L<l, \, l+q\leq L_1\atop{\Delta_1<ml\le \Delta_2\atop{\Delta_1<m(l+q)\le \Delta_2}}}b(l+q)\overline{b(l)}
e\Big(f(l+q, m)-f(l, m)\Big)\right)^{\frac{1}{2}}\nonumber\\
&\ll M^{\frac{1}{2}+\eta}\Bigg(\frac{L}{Q}\sum\limits_{M<m\le M_1}\Bigg(L^{1+\eta}\nonumber\\
&\hspace{20mm}+\sum_{1\leq |q|<Q}\bigg(1-\frac{q}{Q}\bigg)
\sum\limits_{L<l, \, l+q\leq L_1\atop{\Delta_1<ml\le \Delta_2\atop{\Delta_1<m(l+q)\le \Delta_2}}}
b(l+q)\overline{b(l)}e\Big(f(l+q, m)-f(l, m)\Big)\Bigg)^{\frac{1}{2}}\nonumber\\
&\ll X^\eta\Bigg(\frac{X^2}{Q}+\frac{X}{Q}\sum\limits_{1\leq |q|\leq Q}
\sum\limits_{L<l, \, l+q\leq L_1}\bigg|\sum\limits_{M'<m\leq M_1'}e\big(f(l, m, q)\big)\bigg|\Bigg)^{\frac{1}{2}}\,,
\end{align}
where
\begin{equation}\label{M'M1'}
M'=\max{\bigg\{M,\frac{\Delta_1}{l},\frac{\Delta_1}{l+q}\bigg\}}\,,
\quad M_1'=\min{\bigg\{M_1,\frac{\Delta_2}{l},\frac{\Delta_2}{l+q}\bigg\}}\,.
\end{equation}
and
\begin{equation*}
f(l, m, q)=\alpha m^c\Big((l+q)^c \tan^\theta\big(\log (m(l+q))\big)-l^c \tan^\theta\big(\log (ml)\big)\Big)\,.
\end{equation*}
From  \eqref{Conditions2} and \eqref{M'M1'} for the sum  in \eqref{SIIest1} it follows
\begin{equation}\label{M'M1'inM2M}
\left|\begin{array}{ccccc}
L<l, \, l+q\leq L_1\quad\\
M'<m\leq M_1'\quad \quad \;\\
\Delta_1<ml\le \Delta_2 \quad\quad\;\\
\;\Delta_1<m(l+q)\le \Delta_2 \\
(M', M'_1]\subseteq (M, 2M]
\end{array}\right..
\end{equation}
We have
\begin{align}\label{firstderivativem}
\frac{\partial f(l, m, q)}{\partial m}
&=\alpha (q+l)^cm^{c-1}\tan^{\theta-1}\big(\log((q+l)m)\big)\nonumber\\
&\times\Big(c\tan\big(\log((q+l)m)\big)+\theta\sec^2\big(\log((q+l)m)\big)\Big)\nonumber\\
&-\alpha l^cm^{c-1}\tan^{\theta-1}\big(\log(ml)\big)
\Big(c\tan\big(\log(ml)\big))+{\theta}\sec^2\big(\log(ml)\big)\Big)
\end{align}
and
\begin{align}\label{secondderivativem}
\frac{\partial^2f(l, m, q)}{\partial m^2}
&=\alpha (q+l)^cm^{c-2}\tan^{\theta-2}\big(\log((q+l)m)\big)\nonumber\\
&\times\Big(\big(2\theta\sec^2\big(\log((q+l)m)\big)+c^2-c\big)\tan^2\big(\log((q+l)m)\big)\nonumber\\
&+(2c-1)\theta\sec^2\big(\log((q+l)m)\big)\tan\big(\log((q+l)m)\big)\nonumber\\
&+(\theta^2-\theta))\sec^4\big(\log((q+l)m)\big)\Big)\nonumber\\
&-\alpha l^cm^{c-2}\tan^{\theta-2}\big(\log(ml)\big)
\Big(\big(2\theta\sec^2\big(\log(ml)\big)+c^2-c\big)\tan^2\big(\log(ml)\big)\nonumber\\
&+(2c-1)\theta\sec^2\big(\log(ml)\big)\tan\big(\log(ml)\big)+(\theta^2-\theta)\sec^4\big(\log(ml)\big)\Big)\,.
\end{align}
From \eqref{Conditions2}, \eqref{M'M1'inM2M}, \eqref{firstderivativem} and \eqref{secondderivativem} we obtain
\begin{equation}\label{firstderivativemasymp}
\frac{\partial f(l, m, q)}{\partial m}\asymp |\alpha| L^c M^{c-1}
\end{equation}
and
\begin{equation}\label{secondderivativemasymp}
\frac{\partial^2f(l, m, q)}{\partial m^2}\asymp |\alpha| L^c M^{c-2}\,.
\end{equation}
Now \eqref{tau}, \eqref{H}, \eqref{DeltaalphaH2}, \eqref{Conditions2}, \eqref{LMasympX2},
\eqref{SIIest1}, \eqref{M'M1'inM2M}, \eqref{firstderivativemasymp},
\eqref{secondderivativemasymp} and Lemma \ref{GrahamandKolesnik} with  $k=0$  imply
\begin{align*}\label{SIIest2}
S_{II}&\ll X^\eta\Bigg(\frac{X^2}{Q}+\frac{X}{Q}\sum\limits_{1\leq q\leq Q}
\sum\limits_{L<l\leq L_1}\bigg(|\alpha|^{\frac{1}{2}}L^{\frac{c}{2}} M^{\frac{c}{2}}
+|\alpha|^{-1}L^{-c}M^{1-c}\bigg)\Bigg)^{\frac{1}{2}}\\
&\ll X^\eta\Bigg(\frac{X^2}{Q}+X\Big(H^{\frac{1}{2}} L X^{\frac{c}{2}}
+\tau^{-1}L^{1-c}M^{1-c}\Big)\Bigg)^{\frac{1}{2}}\\
&\ll  X^{\frac{17}{18}+\eta}\,.
\end{align*}
This proves the lemma.
\end{proof}

\begin{lemma}\label{SalphaXest} Let $\tau \leq |\alpha| \leq H$.
Then  for the exponential sum denoted by \eqref{Salpha} we have
\begin{equation*}
S(\alpha)\ll  X^{\frac{17}{18}+\eta}\,.
\end{equation*}
\end{lemma}
\begin{proof}
In order to prove the lemma  we will use the formula
\begin{equation}\label{Lambdalog2}
S(\alpha)=S^\ast(\alpha)+\mathcal{O}\Big(X^{\frac{1}{2}+\varepsilon}\Big)\,,
\end{equation}
where
\begin{equation}\label{Sast}
S^\ast(\alpha)=\sum\limits_{\Delta_1<n\leq\Delta_2}\Lambda(n)e\big(\alpha n^c\tan^\theta(\log n)\big)\,.
\end{equation}
Taking into account \eqref{Delta1} and \eqref{Delta2} we have that
\begin{equation}\label{Delta1Delta2inequalities}
\Delta_2<2\Delta_1\,, \quad  2^{-4}X<\Delta_1< 2X\,,  \quad  2^{-3}X<\Delta_2< 2^2X\,,
\end{equation}
thus we choose
\begin{equation}\label{UVZ}
U=2^{-11}X^{\frac{1}{27}}\,,\quad V=2^7X^{\frac{1}{3}}\,,\quad Z=X^{\frac{13}{27}}\,.
\end{equation}
According to Lemma \ref{Heath-Brown}, the sum $S^\ast(\alpha)$
can be decomposed into $O\Big(\log^6\Delta_1\Big)$ sums, each of which is either of Type I
\begin{equation*}
\mathop{\sum\limits_{M<m\le M_1}a_m\sum\limits_{L<l\le L_1}}_{\Delta_1<ml\le \Delta_2}
e\big(\alpha m^cl^c \tan^\theta\big(\log (ml) \big)\big)
\end{equation*}
and
\begin{equation*}
\mathop{\sum\limits_{M<m\le M_1}a_m\sum\limits_{L<l\le L_1}}_{\Delta_1<ml\le \Delta_2}
e\big(\alpha m^cl^c \tan^\theta\big(\log (ml) \big)\big)\log l\,,
\end{equation*}
where
\begin{equation*}
L\ge Z\,,\quad M_1\le 2M\,,\quad L_1\le 2L\,,\quad a_m\ll \tau _5(m)\log \Delta_1
\end{equation*}
or of Type II
\begin{equation*}
\mathop{\sum\limits_{M<m\le M_1}a_m\sum\limits_{L<l\le L_1}}_{\Delta_1<ml\le \Delta_2}b_l
e\big(\alpha m^cl^c \tan^\theta\big(\log (ml) \big)\big)
\end{equation*}
where
\begin{equation*}
U\le L\le V\,,\quad M_1\le 2M\,,\quad L_1\le 2L\,,\quad
a_m\ll \tau _5(m)\log \Delta_1\,,\quad b_l\ll \tau _5(l)\log \Delta_1\,.
\end{equation*}
Using \eqref{Sast}, \eqref{Delta1Delta2inequalities}, \eqref{UVZ}, Lemma \ref{SIest} and  Lemma \ref{SIIest} we deduce
\begin{equation}\label{Sastest}
S^\ast(\alpha)\ll  X^{\frac{17}{18}+\eta}\,.
\end{equation}
Bearing in mind \eqref{Lambdalog2} and  \eqref{Sastest} we establish the statement in the lemma.
\end{proof}

\section{Proof of the Theorem}
\indent

Consider the sum
\begin{equation*}
\Gamma= \sum\limits_{\Delta_1<p_1,p_2,p_3\leq\Delta_2
\atop{|p^c_1\tan^\theta(\log p_1)+  p^c_2\tan^\theta(\log p_2)+ p^c_3\tan^\theta(\log p_3)-N|
<\varepsilon}}\log p_1\log p_2\log p_3\,.
\end{equation*}
The theorem will be proved if we show that $\Gamma\rightarrow\infty$ as $X\rightarrow\infty$.

According to the definition of $\psi(y)$ and the inverse Fourier transformation formula we have
\begin{align}\label{GammaGamma0}
\Gamma\geq\Gamma_0
&= \sum\limits_{\Delta_1<p_1,p_2,p_3\leq\Delta_2}
\psi\big(p^c_1\tan^\theta(\log p_1)+  p^c_2\tan^\theta(\log p_2)+ p^c_3\tan^\theta(\log p_3) -N\big)\log p_1\log p_2\log p_3\nonumber\\
&= \sum\limits_{\Delta_1<p_1,p_2,p_3\leq\Delta_2}\log p_1\log p_2\log p_3\nonumber\\
&\times\int\limits_{-\infty}^{\infty}\Psi(\alpha)e\big(\big(y_1^c\tan^\theta(\log y_1)+y_2^c\tan^\theta(\log y_2)
+y_3^c\tan^\theta(\log y_3)-N\big)\alpha\big)\,d\alpha \nonumber\\
&=\int\limits_{-\infty}^{\infty}S^3(\alpha)\Psi(\alpha)e(-N\alpha)\,d\alpha.
\end{align}
We decompose $\Gamma_0$ in three parts
\begin{equation}\label{Gamma0decomp}
\Gamma_0=\Gamma_1+\Gamma_2+\Gamma_3\,,
\end{equation}
where
\begin{align}
\label{Gamma1}
&\Gamma_1=\int\limits_{-\tau}^{\tau}S^3(\alpha)\Psi(\alpha)e(-N\alpha)\,d\alpha,\\
\label{Gamma2}
&\Gamma_2=\int\limits_{\tau\leq|\alpha|\leq H}S^3(\alpha)\Psi(\alpha)e(-N\alpha)\,d\alpha,\\
\label{Gamma3}
&\Gamma_3=\int\limits_{|\alpha|>H}S^3(\alpha)\Psi(\alpha)e(-N\alpha)\,d\alpha.
\end{align}

\subsection{Estimation of $\mathbf{\Gamma_1}$}

Denote the integrals
\begin{align}
\label{Thetataudef}
&\Theta_\tau=\int\limits_{-\tau}^\tau I^3(\alpha) \Psi(\alpha)e(-N\alpha)\,d\alpha\,,\\
\label{Thetadef}
&\Theta=\int\limits_{-\infty}^\infty I^3(\alpha) \Psi(\alpha)e(-N\alpha)\,d\alpha\,.
\end{align}
For $\Gamma_1$ denoted by  \eqref{Gamma1} we have
\begin{equation}\label{Gamma1ThetaThetatau}
\Gamma_1=(\Gamma_1-\Theta_\tau)+(\Theta_\tau-\Theta)+\Theta\,.
\end{equation}
From \eqref{Gamma1}, \eqref{Thetataudef}, Lemma \ref{Fourier}, Lemma \ref{SIlemma}
and Lemma \ref{3Int} $\textmd{(i)}$, $\textmd{(ii)}$ we find
\begin{align}\label{Gamma1Thetatau}
|\Gamma_1-\Theta_\tau|&\ll\int\limits_{-\tau}^{\tau}
|S^3(\alpha)-I^3(\alpha)||\Psi(\alpha)|\,d\alpha\nonumber\\
&\ll\varepsilon\int\limits_{-\tau}^{\tau}
|S(\alpha)-I(\alpha)|\big(|S(\alpha)|^2+|I(\alpha)|^2\big)\,d\alpha\nonumber\\
&\ll\varepsilon Xe^{-(\log X)^{1/5}}
\Bigg(\int\limits_{-\tau}^{\tau}|S(\alpha)|^2\,d\alpha
+\int\limits_{-\tau}^{\tau}|I(\alpha)|^2\,d\alpha\bigg)\nonumber\\
&\ll\varepsilon X^{3-c}e^{-(\log X)^{1/6}}\,.
\end{align}
We have
\begin{equation}\label{Firstderivative2}
\frac{\partial\alpha y^c\tan^\theta(\log y)\big)}{\partial y}=
\alpha y^{c-1}\Big(c\tan(\log y)+\theta\sec^2(\log y)\Big)\tan^{\theta-1}(\log y)\,.
\end{equation}
Using \eqref{tau}, \eqref{Delta1}, \eqref{Delta2}, \eqref{Int}, \eqref{Thetataudef}, \eqref{Thetadef}, \eqref{Firstderivative2},
Lemma \ref{Fourier} and Lemma \ref{Iest} we deduce
\begin{align}\label{ThetatauTheta}
|\Theta_\tau-\Theta|&\ll\int\limits_{\tau}^{\infty}|I(\alpha)|^3|\Psi(\alpha)|\,d\alpha
\ll\frac{\varepsilon}{X^{3(c-1)}}\int\limits_{\tau}^{\infty}\frac{d\alpha}{\alpha^3}\nonumber\\
&\ll \frac{\varepsilon}{\tau^2X^{3(c-1)}}\ll\frac{\varepsilon X^{3-c}}{\log X}\,.
\end{align}
Bearing in mind \eqref{Thetadef}, \eqref{Gamma1ThetaThetatau}, \eqref{Gamma1Thetatau},
\eqref{ThetatauTheta} and Lemma \ref{Thetaest} we obtain
\begin{equation}\label{Gamma1est}
\Gamma_1\gg \varepsilon X^{3-c}\,.
\end{equation}

\subsection{Estimation of $\mathbf{\Gamma_2}$}

In this subsection we use Cai's \cite{Cai3} argument.

Taking into account Lemma \ref{Fourier} and Lemma \ref{3Int} $\textmd{(iii)}$ we get
\begin{align}\label{InttauH}
\int\limits_{\tau\leq|\alpha|\leq H}|S(\alpha)|^2|\Psi(\alpha)|\,d\alpha
&\ll\varepsilon\int\limits_{\tau}^{1/\varepsilon}|S(\alpha)|^2\,d\alpha
+\int\limits_{1/\varepsilon}^{H}|S(\alpha)|^2\,\frac{d\alpha}{\alpha}\nonumber\\
&\ll\varepsilon\sum\limits_{0\leq n\leq 1/\varepsilon}\int\limits_{n}^{n+1}|S(\alpha)|^2\,d\alpha
+\sum\limits_{1/\varepsilon-1\leq n\leq H}\frac{1}{n}\int\limits_{n}^{n+1}|S(\alpha)|^2\,d\alpha\nonumber\\
&\ll X \log^4X\,.
\end{align}
Denote
\begin{equation}\label{At}
A(t)=\sum\limits_{\Delta_1<n\leq \Delta_2} e\big(t n^c\tan^\theta(\log n)\big)\,.
\end{equation}
From \eqref{Delta1}, \eqref{Delta2} and \eqref{Gamma2} we write
\begin{align}\label{Gamma2est1}
|\Gamma_2|&=\Bigg|\sum\limits_{\Delta_1<p\leq \Delta_2}(\log p)\int\limits_{\tau\leq|\alpha|\leq H}e\big(\alpha p^c\tan^\theta(\log p)\big)
S^2(\alpha)\Psi(\alpha)e(-N\alpha)\,d\alpha\Bigg|\nonumber\\
&\leq\sum\limits_{\Delta_1<p\leq \Delta_2}(\log p)\Bigg|\int\limits_{\tau\leq|\alpha|\leq H}e\big(\alpha p^c\tan^\theta(\log p)\big)
S^2(\alpha)\Psi(\alpha)e(-N\alpha)\,d\alpha\Bigg|\nonumber\\
&\leq(\log X)\sum\limits_{\Delta_1<n\leq \Delta_2}\Bigg|\int\limits_{\tau\leq|\alpha|\leq H}e\big(\alpha n^c\tan^\theta(\log n)\big)
S^2(\alpha)\Psi(\alpha)e(-N\alpha)\,d\alpha\Bigg|\,.
\end{align}
By \eqref{Delta1}, \eqref{Delta2}, \eqref{Gamma2est1} and Cauchy's inequality we deduce
\begin{align}\label{Gamma2est2}
|\Gamma_2|^2
&\leq X(\log X)^2\sum\limits_{\Delta_1<n\leq \Delta_2}\Bigg|\int\limits_{\tau\leq|\alpha|\leq H}e\big(\alpha n^c\tan^\theta(\log n)\big)
S^2(\alpha)\Psi(\alpha)e(-N\alpha)\,d\alpha\Bigg|^2\nonumber\\
&= X(\log X)^2\int\limits_{\tau\leq|y|\leq H}\overline{S^2(y)\Psi(y)e(-Ny)}\,dy
\int\limits_{\tau\leq|t|\leq H}S^2(t)\Psi(t)e(-Nt)A(t-y)\,dt\nonumber\\
&\leq X(\log X)^2\int\limits_{\tau\leq|y|\leq H}|S(y)|^2|\Psi(y)|\,dy
\int\limits_{\tau\leq|t|\leq H}|S(t)|^2|\Psi(t)||A(t-y)|\,dt\,.
\end{align}
We have
\begin{equation}\label{firstderivativet}
\frac{\partial t y^c\tan^\theta(\log y)}{\partial y} =t y^{c-1}\tan^{\theta-1}(\log y)
\Big(c\tan(\log y)+\theta\sec^2(\log y)\Big)
\end{equation}
and
\begin{align}\label{secondderivativet}
\frac{\partial^2t y^c\tan^\theta(\log y)}{\partial y^2} &=t y^{c-2}\tan^{\theta-2}(\log y)
\Big(\big(2\theta\sec^2(\log y)+c^2-c\big)\tan^2(\log y)\nonumber\\
&+(2c-1)\theta\sec^2(\log y)\tan(\log y)+(\theta^2-\theta)\sec^4(\log y)\Big)\,.
\end{align}
From \eqref{Delta1}, \eqref{Delta2}, \eqref{Delta1Delta2inequalities}, \eqref{firstderivativet} and  \eqref{secondderivativet} it follows
\begin{equation}\label{firstderivativetest}
\frac{\partial t y^c\tan^\theta(\log y)}{\partial y}\asymp  |t|\Delta_1^{c-1} \quad \mbox{ for } \quad y\in [\Delta_1, \Delta_2]
\end{equation}
and
\begin{equation}\label{secondderivativetest}
\frac{\partial^2t y^c\tan^\theta(\log y)}{\partial y^2}\asymp  |t|\Delta_1^{c-2} \quad \mbox{ for } \quad y\in [\Delta_1, \Delta_2]\,.
\end{equation}
Now \eqref{Delta1}, \eqref{Delta2}, \eqref{Delta1Delta2inequalities}, \eqref{At}, \eqref{firstderivativetest}, \eqref{secondderivativetest}
and Lemma \ref{GrahamandKolesnik} with  $k=0$  imply
\begin{equation}\label{Atest}
A(t)\ll \min\Big(|t|^{\frac{1}{2}}X^{\frac{c}{2}}+|t|^{-1}X^{1-c}, X\Big)\,.
\end{equation}
Using  \eqref{H}, \eqref{InttauH}, \eqref{Atest}, Lemma \ref{Fourier} and Lemma \ref{SalphaXest} we write
\begin{align}\label{IntPsiAt}
&\int\limits_{\tau\leq|t|\leq H}|S(t)|^2|\Psi(t)||A(t-y)|\,dt\nonumber\\
&\ll\int\limits_{\tau\leq|t|\leq H\atop{|t-y|\leq X^{-c}}}|S(t)|^2|\Psi(t)||A(t-y)|\,dt
+\int\limits_{\tau\leq|t|\leq H\atop{X^{-c}<|t-y|\leq2H}}|S(t)|^2|\Psi(t)||A(t-y)|\,dt\nonumber\\
&\ll \varepsilon  X\int\limits_{\tau\leq|t|\leq H\atop{|t-y|\leq X^{-c}}}|S(t)|^2\,dt
+\int\limits_{\tau\leq|t|\leq H\atop{X^{-c}<|t-y|\leq2H}}|S(t)|^2|\Psi(t)|
\Bigg(|t-y|^{\frac{1}{2}}X^{\frac{c}{2}}+\frac{X^{1-c}}{|t-y|}\Bigg)\,dt\nonumber\\
&\ll \varepsilon X\max\limits_{\tau\leq|t|\leq H}|S(t)|^2\int\limits_{|t-y|\leq X^{-c}}\,dt
+H^{\frac{1}{2}}X^{\frac{c}{2}}\int\limits_{\tau\leq|t|\leq H}|S(t)|^2|\Psi(t)|\,dt\nonumber\\
&+\varepsilon X^{1-c}\max\limits_{\tau\leq|t|\leq H}|S(t)|^2\int\limits_{X^{-c}<|t-y|\leq2H}\frac{1}{|t-y|}\,dt\nonumber\\
&\ll \varepsilon  X^{1-c+\eta}\max\limits_{\tau\leq|t|\leq H}|S(t)|^2
+H^{\frac{1}{2}}X^{\frac{c}{2}+1+\eta}\nonumber\\
&\ll \varepsilon X^{\frac{26}{9}-c+\eta}\,.
\end{align}
Bearing in mind \eqref{varepsilon}, \eqref{InttauH}, \eqref{Gamma2est2} and \eqref{IntPsiAt}  we  obtain
\begin{equation}\label{Gamma2est3}
\Gamma_2\ll \varepsilon^\frac{1}{2}X^{\frac{22}{9}-\frac{c}{2}+\eta}\ll \frac{\varepsilon X^{3-c}}{\log X}\,.
\end{equation}

\subsection{Estimation of $\mathbf{\Gamma_3}$}

Using \eqref{varepsilon}, \eqref{H}, \eqref{Salpha}, \eqref{Gamma3}, Lemma \ref{Fourier} and choosing $k=[\log X]$ we find
\begin{align}\label{Gama3est}
\Gamma_3&\ll \int\limits_{H}^{\infty}|S(\alpha)|^3|\Psi(\alpha)|\,d\alpha\nonumber\\
&\ll X^3\int\limits_{H}^{\infty}\frac{1}{\alpha}\bigg(\frac{k}{2\pi \alpha\varepsilon/8}\bigg)^k\,d\alpha\nonumber\\
&=X^3\bigg(\frac{4k}{\pi\varepsilon H}\bigg)^k\ll1.
\end{align}

\subsection{The end of the proof}

Bearing in mind \eqref{GammaGamma0}, \eqref{Gamma0decomp}, \eqref{Gamma1est}, \eqref{Gamma2est3} and
\eqref{Gama3est} we establish that
\begin{equation}\label{Gammaest}
\Gamma\gg \varepsilon X^{3-c}\,.
\end{equation}
Now \eqref{varepsilon} and \eqref{Gammaest}
imply that $\Gamma\rightarrow\infty$ as $X\rightarrow\infty$.

The proof of the Theorem is complete.

\vskip20pt
\footnotesize
\begin{flushleft}
S. I. Dimitrov\\
Faculty of Applied Mathematics and Informatics\\
Technical University of Sofia \\
8, St.Kliment Ohridski Blvd. \\
1756 Sofia, BULGARIA\\
e-mail: sdimitrov@tu-sofia.bg\\
\end{flushleft}


\begin{thebibliography}{100}

\bibitem{Baker} R. Baker, {\it Some diophantine equations and inequalities with primes},
Funct. Approx. Comment. Math., \textbf{64} (2), (2021), 203 -- 250.

\bibitem{Baker-Weingartner} R. Baker, A. Weingartner, {\it A ternary diophantine inequality over primes},
Acta Arith., {\bf162}, (2014), 159 -- 196.

\bibitem{Cai1} Y. Cai, {\it On a diophantine inequality involving prime numbers},
Acta Math. Sinica, Chinese Series, {\bf39}, (1996), 733 -- 742.

\bibitem{Cai2} Y. Cai, {\it On a diophantine inequality involving prime numbers III},
Acta Math. Sinica, English Series, {\bf15}, (1999), 387 -- 394.

\bibitem{Cai3} Y. Cai, {\it A ternary Diophantine inequality involving primes},
Int. J. Number Theory, {\bf14}, (2018), 2257 -- 2268.

\bibitem{Cao-Zhai}  X. Cao, W. Zhai, {\it A Diophantine inequality with prime numbers},
Acta Math. Sinica, Chinese Series, {\bf45}, (2002), 361 -- 370.

\bibitem{Dimitrov} S. I. Dimitrov, A logarithmic inequality involving prime numbers,
\emph{Proc. Jangjeon Math. Soc.}, \textbf{24}, 3, (2021), 403 -- 416.

\bibitem{Graham-Kolesnik} S. W. Graham, G. Kolesnik, {\it Van der Corput's Method of Exponential Sums},
Cambridge University Press, New York, (1991).

\bibitem{Heath} D. R. Heath-Brown, {\it Prime numbers in short intervals and a generalized Vaughan identity},
Canad. J. Math., \textbf{34}, (1982), 1365 -- 1377.

\bibitem{Iwaniec-Kowalski} H. Iwaniec, E. Kowalski, {\it Analytic number theory},
Colloquium Publications, \textbf{53}, Amer. Math. Soc., (2004).

\bibitem{Ku-Ne} A. Kumchev, T. Nedeva,
{\it On an equation with prime numbers}, Acta Arith., {\bf 83},
(1998), 117 -- 126.

\bibitem{Kumchev} A. Kumchev, {\it A diophantine inequality involving prime powers},
Acta Arith., {\bf 89}, (1999), 311 -- 330.

\bibitem{Shapiro} I. Piatetski-Shapiro, {\it On a variant of the Waring-Goldbach problem},
Mat. Sb., {\bf30}, (1952), 105 -- 120, (in Russian).

\bibitem{Segal}B. I. Segal, {\it On a theorem analogous to Waring's theorem},
Dokl. Akad. Nauk SSSR (N. S.), {\bf2}, (1933), 47 -- 49, (in Russian).

\bibitem{Titchmarsh}E. Titchmarsh, {\it The Theory of the Riemann Zeta-function} (revised by D. R.
Heath-Brown), Clarendon Press, Oxford (1986).

\bibitem{Tolev1}D. I. Tolev, {\it On a diophantine inequality involving prime numbers},
Acta Arith., {\bf61}, (1992), 289 -- 306.


\end{thebibliography}
\end{document}